\documentclass{amsart}

\usepackage{amsmath,amssymb,amsthm,amsfonts}
\usepackage{hyperref}
\usepackage{appendix}
\usepackage{graphicx}

\usepackage{color}

\newtheorem{lemma}{Lemma}[section]
\newtheorem{theorem}{Theorem}[section]

\newtheorem{remark}{Remark}[section]
\newtheorem{corollary}{Corollary}[section]

\numberwithin{equation}{section}

\begin{document}
\title[Boundary layer for nonhomogeneous incompressible fluids]{Boundary layer associated with a class of 3D nonlinear plane parallel channel flows for nonhomogeneous incompressible Navier-Stokes equations}
\author[Shijin Ding, Zhilin Lin, Dongjuan Niu]{Shijin Ding, Zhilin Lin$^*$, Dongjuan Niu}
\address[S. Ding]{South China Research Center for Applied Mathematics and Interdisciplinary Studies, South China Normal University,
Guangzhou, 510631, China}\address{School of Mathematical Sciences, South China Normal University,
Guangzhou, 510631, China}
\email{dingsj@scnu.edu.cn}
\address[Corresponding author: Z. Lin]{School of Mathematical Sciences, South China Normal University,
Guangzhou, 510631, China}
\email{zllin@m.scnu.edu.cn}
\address[D. Niu]{School of Mathematical Sciences, Capital Normal University,
Beijing, 100048, China}
\email{djniu@cnu.edu.cn}
\date{\today}

\begin{abstract}
In this paper, we establish the mathematical validity of the Prandtl boundary layer theory for a class of nonlinear plane parallel flows of nonhomogeneous incompressible Navier-Stokes equations. The convergence is shown under various Sobolev norms, including the physically important space-time uniform norm, as well as the $L^\infty(H^1)$ norm. It is mentioned that the mathematical validity of the Prandtl boundary layer theory for nonlinear plane parallel flow is generalized to the nonhomogeneous case.
\end{abstract}

\maketitle

\vspace{-5mm}

\section{Introduction}

In this paper, we consider the boundary layer of nonlinear plane parallel channel flows for the nonhomogeneous incompressible fluids in a three-dimensional slab domain, periodic in horizontal $x$ and $y$ directions, $\Omega=\mathbb{T}^2 \times [0,1](\mathbb{T}=[0,L])$ with the boundaries $\partial{\Omega}=\{z=i\}, i=0,1$. The motion of the incompressible fluids in $\Omega$ is governed by the following nonhomogeneous incompressible Navier-Stokes equations
\begin{equation}\label{1.1}
\left \{
\begin{array}{lll}
\partial_t \rho^\varepsilon +\mathrm{div}(\rho^\varepsilon u^\varepsilon)=0,\\
\partial_t (\rho^\varepsilon u^\varepsilon)+\mathrm{div}(\rho^\varepsilon u^\varepsilon \otimes u^\varepsilon)-\varepsilon \Delta u^\varepsilon +\nabla p^\varepsilon =\rho^\varepsilon f,\\
\mathrm{div}\ u^\varepsilon =0,
\end{array}
\right.
\end{equation}
where $u^\varepsilon(t;x,y,z), \rho^\varepsilon(t;x,y,z), p^\varepsilon(t;x,y,z)$ and $f\in\mathbb{R}^3$ are the velocity fields, density, pressure and external force, respectively. The positive constant $\varepsilon$ is the viscosity coefficient.

It is well known that the Navier-Stokes equations are equipped with the following no-slip boundary condition and initial data
\begin{equation}\label{1.2}
\left \{
\begin{array}{lll}
u^\varepsilon=0 \ \ \textrm{on} \ \ {\partial{\Omega}},\\
(\rho^\varepsilon,u^\varepsilon)|_{t=0}=(\rho_0,u_0), \  \rho_0 >0.
\end{array}
\right.
\end{equation}

Letting $\varepsilon = 0$, we arrive at the following nonhomogeneous incompressible Euler equations
\begin{equation}\label{1.3}
\left \{
\begin{array}{lll}
\partial_t \rho^0 +\mathrm{div}(\rho^0 u^0)=0,\\
\partial_t (\rho^0 u^0)+\mathrm{div}(\rho^0 u^0 \otimes u^0)+\nabla p^0 =\rho^0 f,\\
\mathrm{div}\ u^0 =0,
\end{array}
\right.
\end{equation}
with the following no-penetration boundary conditions and the same initial data
\begin{equation}\label{1.4}
\left \{
\begin{array}{lll}
u^0 \cdot n=0 \ \ \textrm{on} \ \ {\partial{\Omega}},\\
(\rho^0,u^0)|_{t=0}=(\rho_0,u_0), \  \rho_0 >0,
\end{array}
\right.
\end{equation}
in which $n$ is the unit outward normal to the boundaries.

In addition, we suppose that the initial datum of the density is away from vacuum, that is, there holds that
\begin{equation}\label{1.5}
\rho_0 \geq c_0>0
\end{equation}
for some constant $c_0$. Then by the classical theory of Navier-Stokes equations and Euler equations, one easily deduces that
\begin{equation}\label{1.6}
\rho^\varepsilon(t;x,y,z) \geq c_0>0 , \   \rho^0(t;x,y,z) \geq c_0>0
\end{equation}
for any time $t \geq 0$.

In this paper, our aim is to study the boundary layer for \textbf{(nonlinear) plane parallel channel flows}. In homogeneous case, this problem had been studied in \cite{Mazzucato,XWang}. In this paper, we will investigate the nonhomogeneous case. Precisely, we intend to look for the solutions of the equation (\ref{1.1}) of the form
\begin{equation}\label{1.61}
\rho^\varepsilon=\rho^\varepsilon(t;z), \ u^\varepsilon=(u^\varepsilon_1(t;z),u^\varepsilon_2(t;x,z),0)
\end{equation}
in an infinitely long horizontal channel. Moreover we suppose that the domain is periodic in horizontal $x$ and $y$ directions. Therefore we reduce to consider the problem in the domain $Q=[0,L]^2 \times [0,1]$, in which $L$ is the horizontal period, see Figure \ref{1.611} for instance. Obviously, the flows of the form (\ref{1.61}) automatically satisfy the divergence-free condition, i.e., $\textrm{div}\ u^\varepsilon=0$.

\begin{figure}[htbp]\label{1.611}
  \centering
  \includegraphics[width=0.66\textwidth]{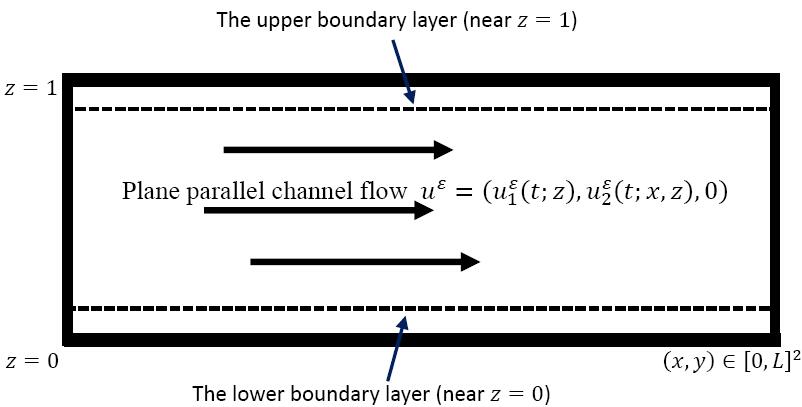}
  \caption{ The plane parallel channel flow in $Q=[0,L]^2 \times [0,1]$}
\end{figure}

The special structure of the solution is preserved by both Navier-Stokes equations and Euler equations if the initial data $(\rho_0,u_0)$ satisfy the same ansatz, i.e.,
\begin{equation}\label{1.62}
u^\varepsilon|_{t=0}=u_0=(a(z),b(x,z),0), \ \rho^\varepsilon|_{t=0}=\rho_0(z).
\end{equation}
Indeed, we have
$$\rho^\varepsilon\equiv \rho_0(z).$$
The equations for $u^\varepsilon_1$ is a heat equation (only depends on $z$) and that for $u^\varepsilon_2$ is a linear equation (depends on $x,z$). Therefore, in (\ref{1.64}), one can easily check that the structure of the flow can be preserved, i.e.,
$$u^\varepsilon=(u^\varepsilon_1(t;z),u^\varepsilon_2(t;x,z),0).$$
For more about the symmetry of solution to the Navier-Stokes equations in homogeneous fluids, see \cite{bardos} for details.

Let us denote the solution of Navier-Stokes equations by $u^\varepsilon$ with the viscosity $\varepsilon$ and that of Euler equations by $u^0$. For the Navier-Stokes equations, we impose the following boundary conditions
\begin{equation}\label{1.63}
u^\varepsilon|_{z=i}=\alpha^i(t;x),
\end{equation}
where $\alpha^i(t;x)=(\beta^i_1(t),\beta^i_2(t;x),0), i=0,1$.

It is easy to see that the solutions (\ref{1.61}) satisfy
\begin{equation}\label{1.64}
\left \{
\begin{array}{lll}
\partial_t\rho^\varepsilon =0,\\
\rho^\varepsilon \partial_t u^\varepsilon_1-\varepsilon \partial_z^2 u^\varepsilon_1=\rho^\varepsilon f_1,\\
\rho^\varepsilon \partial_t u^\varepsilon_2-\varepsilon \Delta_{x,z} u^\varepsilon_2+\rho^\varepsilon u^\varepsilon_1\partial_x u^\varepsilon_2=\rho^\varepsilon f_2,
\end{array}
\right.
\end{equation}
in $(x,z) \in [0,L] \times [0,1]$. Note that the plane parallel flows are three-dimensional actually. In addition, we will assume that the initial data, boundary data satisfy certain compatibility conditions. Recall that the zero-order compatibility conditions with the form
\begin{equation}\label{1.65}
\alpha^i(0;x)=u_0(x,i), \ \ i=0,1,
\end{equation}
and the first-order compatibility conditions
\begin{equation}\label{1.66}
\left \{
\begin{array}{lll}
\partial_t\rho^\varepsilon(0;i)=0,\\
\rho_0(i)\partial_t \beta^i_1(0)-\varepsilon \partial_z^2 a(i)=\rho_0(i)f(0;i),\\
\rho_0(i)\partial_t \beta^i_2(0;x)+\rho_0(x,i)\beta^i_1(0)\partial_x \beta^i_1(0;x)\\
\quad \quad \quad \quad \quad \quad \quad \quad -\varepsilon \Delta_{x,z} b(x,i)=\rho_0(i)f_2(0;x,i),
\end{array}
\right.
\end{equation}
where $\Delta_{x,z}:=\partial_{xx}+\partial_{zz}$ and $i=0,1$.

The well-posedness of the system (\ref{1.64}) can be easily obtained because of the weak coupling in (\ref{1.64}). For instance, one can get that $(\rho^\varepsilon,u^\varepsilon) \in L^\infty(H^1) \times L^\infty(H^1)$ and $( \rho^\varepsilon_t,\sqrt{\rho^\varepsilon}u^\varepsilon_t ) \in L^\infty(L^2) \times L^\infty(L^2)$ provided that $\rho_0 \in H^1 \cap L^\infty, u_0 \in H^2 \cap H^1, \alpha^i \in H^1, f\in L^\infty(0,T;H^1)$. We do not address this point in details here, and refer for example to \cite{Kim,Kim1,Simon,Wen} for interested readers.

By formally taking $\varepsilon = 0$, the Navier-Stokes equations become the corresponding Euler equations. Under the plane-parallel assumption, the Euler system reduces to the following weakly nonlinear equations
\begin{equation}\label{1.67}
\left \{
\begin{array}{lll}
\partial_t\rho^0 =0,\\
\rho^0\partial_t  u^0_1=\rho^0f_1,\\
\rho^0\partial_t  u^0_2+\rho^0u^0_1 \partial_x u^0_2=\rho^0f_2.
\end{array}
\right.
\end{equation}
We take the same initial data for both Euler and Navier-Stokes equations
\begin{equation}\label{1.68}
(\rho^0,u^0)|_{t=0}=(\rho_0(z),u_0(x,z)).
\end{equation}
Moreover, the solutions of (\ref{1.67}) -- (\ref{1.68}) can be obtained by solving a simple ODE and two linear transport equations thanks to $\rho_0 \geq c_0>0$.
Therefore the solutions are regular provided the initial data are regular enough. For example, if $(\rho_0,u_0) \in H^m(\Omega)\times H^m(\Omega)$ and $f \in L^\infty(0,T;H^m)$ for $m>5$, then $(\rho^0,u^0) \in C(0,T;H^m(\Omega))\times C(0,T;H^m(\Omega))$. The interested readers can see \cite{Marsden,Veiga1,Veiga2} for more details.

It should be pointed out that our aim is to justify the validity of the boundary layer expansion but not the regularity of the solutions. Therefore the solutions in our assumptions are regular enough, if we need.

Let us mention that
\begin{equation}\label{1.68a}
\rho^\varepsilon(t;z)=\rho^0(t;z)\equiv \rho_0(z),
\end{equation}
which is resulted from (\ref{1.64}) and (\ref{1.67}) with the same initial data (\ref{1.68}). Formally, (\ref{1.68a}) implies that there is no difference between the viscous density and ideal density, therefore there is no boundary layer for density. Indeed, we can verify this fact by performing the boundary layer expansions, see Section \ref{sec2}.

Now let us introduce some related results. As we know, the study of the behavior for the fluid with small viscosity (or large Reynold number) is an important topic in mathematics and physics. In 1905, Prandtl \cite{Prandtl} first introduced the concept of boundary layers and deduced the Prandtl equations with the no-slip boundary condition, which adheres to the strong boundary layer. According to the idea of Prandtl \cite{Prandtl}, there is a thin boundary layer of width of the order $\sqrt{\varepsilon}$ near the boundary. More precisely, the solutions to the Navier-Stokes equations with no-slip boundary conditions is expected to take the form
\begin{equation}\label{1.7}
\textrm{Navier--Stokes}\simeq \textrm{Euler}+\textrm{Prandtl \ layer}+O(\sqrt{\varepsilon}).
\end{equation}
Therefore, there are at least two fundamental problems to be investigated:

(a) The well-posedness of the Prandtl equations;

(b) The justification of (\ref{1.7})(or the validity of the boundary layer expansion).

For the Prandtl equation, there are lots of results to deal with the well-posedness or ill-posedness. Up to now, the well-posedness of Prandtl equation was proved only in some special functional spaces. As early as in 1963, Oleinik firstly obtained the local existence of classical solutions in 2D under the monotonic assumption on the tangential velocity, see \cite{Oleinik,Oleinik1} for instance. In the works of Oleinik \cite{Oleinik,Oleinik1}, she introduced the Crocco transform to reduce the Prandtl equation to some classical parabolic equations, which can be solved by some standard methods of PDEs. The well-posedness in the Sobolev spaces by applying energy method was obtained independently in \cite{Alexander,Masmoudi}. By imposing a favorable condition on the pressure, Xin and Zhang \cite{Xin2} proved the global existence of the Prandtl equation. The ill-posedness of Prandtl equation was obtained by G\'{e}rard-Varet and Dormy \cite{Gerard1}. Moreover, some of these results are generalized to the 3D case, see \cite{CLiu,CLiu1} for details. Recently, the well-posedness of the MHD boundary layer equations in Sobolev spaces without monotonicity was shown by Liu, Xie and Yang \cite{CLiu2}.

To justify the validity of the Prandtl expansion, one needs to study the convergence from the viscous solution to the inviscid solution as $\varepsilon \to 0$. However, there are few results about this topic. For the steady case, Guo and Nguyen \cite{YGuo2} proved the validity of the steady boundary layer expansion on a moving boundary and Iyer \cite{Iyer1} studied the similar problem over a rotating disk. Later, G\'{e}rard-Varet and Maekawa \cite{Gerard2} obtained the Sobolev stability of Prandtl expansions for the steady Navier-Stokes equations. Recently, without the moving boundary condition of \cite{YGuo2}, Guo and Iyer \cite{YGuo1} proved the validity of steady Prandtl layer expansion under an assumption for the normal velocity of first order Euler correction.

For the time-dependent case, Sammartino and Caflisch \cite{Sammartino1,Sammartino2} obtained the local existence of analytic solutions to the Prandtl equations and a rigorous theory on the stability of boundary layers in incompressible flow confined in the half space with analytic data. Mazzucato, Niu and Wang \cite{Mazzucato} established the validity of the boundary layer theory for the 3D plane parallel channel flows. With the  similar ideas, Han, Mazzucato, Niu and Wang \cite{Han} proved the validity of the boundary layer expansion for  nonlinear pipe flow. Maekawa \cite{Maekawa} studied the inviscid limit problem of the vorticity equations. For the analytic data, Wang, Wang and Zhang \cite{Wang} developed an energy method to justify the zero-viscosity limit for the analytic data in $\mathbb{R}^2_+$. With the basic idea and methods, Fei, Tao and Zhang \cite{Fei} considered the zero-viscosity limit of the incompressible Navier-Stokes equations with no-slip boundary condition in $\mathbb{R}^3_+$ and generalized Maekawa's result \cite{Maekawa} to 3D case. Recently, in homogeneous case, the convergence results without the compatibility conditions for some symmetric flows were established by Gie et al. \cite{Gie}. In addition, the validity of the boundary layer theory for incompressible MHD was obtained by Liu, Xie and Yang \cite{CLiu3}.

In this paper, we study the validity of the Prandtl theory associated with a special type plane parallel channel flows for nonhomogeneous Navier-Stokes equations. More about this special type flows for homogeneous fluids can be found in \cite{XWang}. Therefore one should also construct the approximate solution for viscous solutions by the boundary layer expansion. Our main result (see Theorem \ref{thm1}) provides some error bounds for the approximation of the Navier-Stokes equations given by the Euler solution plus the upper boundary and lower boundary correctors. Therefore the convergence rates in vanishing viscosity limit are obtained. Moreover, the optimal convergence rates in Sobolev norm $L^2$ were also obtained. Compared with the results in \cite{Gie}, since we impose the compatibility conditions, our results provide the convergence rates of the higher order for error solutions.
In order to obtain the error estimates and convergence rates, we need to establish some weighted estimates (or decay estimates) for the correctors. To ensure that the convergence rates hold true from two parts: the remainders and the weighted estimates. However, the convergence rates can not be improved in some sense due to the conditions of known data and structures of the remainders.

Now, our main results are stated as follows.
\begin{theorem}\label{thm1}
Suppose that $u_0 \in H^m(\Omega), \rho_0 \in H^m(\Omega)$ and there exists $c_0>0$ such that $c_0 \leq \rho_0 $. In addition, assume that the external force and boundary data satisfy $f \in L^\infty(0,T;H^m(\Omega)), \beta^i \in H^2(0,T; H^m(\partial{\Omega})),i=0,1, m>5$ and the compatibility conditions (\ref{1.65})-(\ref{1.66}). Then there exist positive constants $C>0$, independent of $\varepsilon$, such that for any solution $(\rho^\varepsilon,u^\varepsilon)$ of (\ref{1.64}) with the initial values $(\rho_0,u_0)$ and boundary values $\beta^i$,
\begin{equation}\label{1.73}
\Vert u^\varepsilon-\tilde{u}^a \Vert_{L^\infty(0,T;L^2(\Omega))} \leq C \varepsilon^{\frac{3}{4}},
\end{equation}
\begin{equation}\label{1.74}
\Vert u^\varepsilon-\tilde{u}^a \Vert_{L^\infty(0,T;H^1(\Omega))} \leq C \varepsilon^{\frac{1}{4}},
\end{equation}
\begin{equation}\label{1.75}
\Vert u^\varepsilon-\tilde{u}^a \Vert_{L^\infty((0,T) \times \Omega))} \leq C \sqrt{\varepsilon},
\end{equation}
where $\tilde{u}^a$ are defined by (\ref{3.1}) in Section \ref{sec3}.
\end{theorem}
Meanwhile, we have the following optimal convergence rate result.
\begin{corollary}\label{cor1}
Under the assumptions of Theorem \ref{thm1}, the following optimal convergence rate holds
\begin{equation}\label{1.77}
\begin{array}{lll}
C \varepsilon^{\frac{1}{4}} \leq \Vert u^\varepsilon-u^0 \Vert_{L^\infty(0,T;L^2(\Omega))} \leq C \varepsilon^{\frac{1}{4}},
\end{array}
\end{equation}
where $u^0$ is the solution of Problem (\ref{1.67})-(\ref{1.68}) and the constants $C>0$ are independent of $\varepsilon$.
\end{corollary}
In addition, similar to that in \cite{XWang}, the following Kato-type conclusion also holds.
\begin{corollary}\label{cor2}
Under the assumptions of Theorem \ref{thm1}, there exists positive constants $C>0$ independent of $\varepsilon$ such that for any $\delta \in (0,1)$ such that $\delta/\varepsilon \to \infty$ as $\varepsilon \to 0$,
\begin{equation}\label{1.79}
\begin{array}{lll}
\Vert u^\varepsilon-u^0 \Vert_{L^\infty(0,T;H^1(\Omega^\delta))} \leq C \varepsilon^{\frac{1}{4}},
\end{array}
\end{equation}
\begin{equation}\label{1.790}
\begin{array}{lll}
\Vert u^\varepsilon-u^0 \Vert_{L^\infty((0,T) \times \Omega^\delta))} \leq C \sqrt{\varepsilon},
\end{array}
\end{equation}
where $\Omega^\delta:=[0,L]\times [\delta,1-\delta]$.
\end{corollary}

\begin{remark}\label{rk1/4}
The results of Corollary \ref{cor1} are straightforward from the fact that
$$\Vert (\theta^0_1,\theta^0_2) \Vert_{L^\infty (0,T;L^2(\Omega_\infty))} \thickapprox \varepsilon^{\frac{1}{4}},$$
where $(\theta^0_1,\theta^0_2)$ are determined in Section \ref{sec2} and $\Omega_\infty=[0,L]\times [0,\infty)$.
\end{remark}

\begin{remark}
The proof of Corollary \ref{cor2} follows from the error estimates and it is similar to \cite{XWang}. For simplicity, we omit it here.
\end{remark}

\begin{remark}
It is well known that the boundary layer is  resulted from the mismatch of the boundary conditions for the velocity fields.
In other words, for the nonhomogeneous fluids, when we take the above special form of (\ref{1.61}), one can deduce that $\rho^\varepsilon=\rho^0\equiv\rho_0(z)$. It implies that there is no boundary layer for density, which is very reasonable.
\end{remark}

The rest of this paper is organized as follows. Section \ref{sec2} is devoted to formal asymptotic expansion of this type of flows at small viscosity. In Section \ref{sec3}, we will construct an approximate solution to Navier-Stokes equations utilizing the solution to the Prandtl-type effective systems and the solution to Euler equations. The main estimates and convergence rates are obtained in Section \ref{sec4}. Some decay estimates for the correctors are provided in Appendix \ref{ap}.

Throughout our paper, we use $C$ to denote a generic constant, independent of $\varepsilon$, that may depend on some initial data. In addition, we denote
$$\langle Z \rangle:=\sqrt{1+|Z|^2}$$
and
$$\langle Z^u \rangle:=\sqrt{1+|Z^u|^2},$$
which will be used later in this paper.

\section{The Prandtl-type effective equations for the correctors}\label{sec2}

The approach to a rigorous boundary layer analysis that we take is to derive the equations for the correctors, which is the difference between the Navier-Stokes solutions $(\rho^\varepsilon,u^\varepsilon,0)$ and the Euler solution $(\rho^0,u^0,0)$, where we have taken $p^\varepsilon=p^0 \equiv 0$ without loss of generality. First we recall that
$$\rho^0(t;z)\equiv \rho_0(z).$$
 We consider the approximate solutions as follows
\begin{equation}\label{2.1}
\left \{
\begin{array}{lll}
\rho^{a}&:=\rho^{\textrm{ou}}(t;z)+\sigma^0\left(t;\frac{z}{\sqrt{\varepsilon}}\right)
+\sigma^{u,0}\left(t;\frac{1-z}{\sqrt{\varepsilon}}\right),\\
u^{a}_1&:=u^{\textrm{ou}}_1(t;z)+\theta^0_1\left(t;\frac{z}{\sqrt{\varepsilon}}\right)
+\theta^{u,0}_1\left(t;\frac{1-z}{\sqrt{\varepsilon}}\right),\\
u^{a}_2&:=u^{\textrm{ou}}_2(t;x,z)+\theta^0_2\left(t;x,\frac{z}{\sqrt{\varepsilon}}\right)
+\theta^{u,0}_2\left(t;x,\frac{1-z}{\sqrt{\varepsilon}}\right),
\end{array}
\right.
\end{equation}
where $(\rho^{\text{ou}},u^{\text{ou}})$ are the outer solutions, and $(\sigma^0,\theta^0_1,\theta^{0}_2),$ $(\sigma^{u,0},\theta^{u,0}_1,\theta^{u,0}_2)$ are the lower correctors and upper correctors, respectively. The correctors satisfy
\begin{equation}\label{2.2}
\left \{
\begin{array}{lll}
(\sigma^0,\theta^0_1,\theta^{0}_2) \to (0,0,0) \ \ &\textrm{as} \ Z:=\frac{z}{\sqrt{\varepsilon}} \to \infty,\\
(\sigma^{u,0},\theta^{u,0}_1,\theta^{u,0}_2)\to (0,0,0)  \ \ &\textrm{as} \ Z^u:=\frac{1-z}{\sqrt{\varepsilon}} \to \infty.
\end{array}
\right.
\end{equation}
It is easy to see that $(\sigma^0,\theta^0_1,\theta^{0}_2),(\sigma^{u,0},\theta^{u,0}_1,\theta^{u,0}_2)$ are defined in $\Omega_\infty:=[0,L]\times [0,\infty).$ Then the outer solutions and correctors satisfy respectively:

\textbf{(I) The outer solution:}

The outer solution $(\rho^{\text{ou}},u^{\text{ou}})$ satisfy the Euler equations (\ref{1.67}) with the initial data
\begin{equation}\label{2.21}
(\rho^0,u^0)|_{t=0}=(\rho_0,u_0).
\end{equation}
The uniqueness of the solutions to (\ref{1.67}) yields that $\rho^{\text{ou}}\equiv \rho^0\equiv\rho_0(z), u^{\text{ou}}\equiv u^0$.

\textbf{(II) The lower corrector $(\sigma^0,\theta^0_1,\theta^{0}_2)$:}

For this part, first we note that for the density, we have
$$\partial_t \sigma^0=0, \ \ \sigma^0|_{t=0}=0.$$
Therefore we have $\sigma^0\equiv 0$, which implies that there is no boundary layer for the density.

Furthermore, the lower corrector $(\sigma^0,\theta^0_1,\theta^{0}_2)$ satisfies
\begin{equation}\label{2.3}
\left \{
\begin{array}{lll}
\rho_0(0)\partial_t \theta^0_1 -\partial_{ZZ} \theta^0_1=0,\\
\rho_0(0)\partial_t \theta^0_2 -\partial_{ZZ} \theta^0_2+(u^0_1(t;0)+\theta^0_1)\rho_0(0)\partial_x \theta^0_2 \\
\quad \quad \quad \quad \quad +\rho_0(0)\partial_x u^0_2(t;x,0)\theta^0_1=0,\\
(\theta^0_1,\theta^0_2)|_{Z=0}=(\beta^0_1(t)-u^0_1(t;0),\beta^0_2(t;x)-u^0_2(t;x,0)),\\
(\theta^0_1,\theta^0_2)|_{Z=\infty}=(0,0),\\
(\theta^0_1,\theta^{0}_2)|_{t=0}=(0,0).
\end{array}
\right.
\end{equation}
Here we have used the fact that $\rho^0(t;z)\equiv \rho_0(z)$.

\textbf{(III) The upper corrector $(\sigma^{u,0},\theta^{u,0}_1,\theta^{u,0}_2)$:}

Similarly, we have $\sigma^{u,0}\equiv 0$. Therefore, the upper corrector $(\sigma^{u,0},\theta^{u,0}_1,\theta^{u,0}_2)$ satisfies
\begin{equation}\label{2.4}
\left \{
\begin{array}{lll}
\rho_0(1)\partial_t \theta^{u,0}_1 -\partial_{Z^u Z^u} \theta^{u,0}_1=0,\\
\rho_0(1)\partial_t \theta^{u,0}_2 +(u^0_1(t;1)+\theta^{u,0}_1)\rho_0(1)\partial_x \theta^{u,0}_2 \\
\quad+\rho_0(1)\partial_x u^0_2(t;x,1)\theta^{u,0}_1-\partial_{Z^uZ^u} \theta^{u,0}_2=0,\\
(\theta^{u,0}_1,\theta^{u,0}_2)|_{Z^u=0}=(\beta^1_1(t)-u^0_1(t;1),\beta^1_2(t;x)-u^0_2(t;x,1)),\\
(\theta^{u,0}_1,\theta^{u,0}_2)|_{Z^u=\infty}=(0,0),\\
(\theta^{u,0}_1,\theta^{u,0}_2)|_{t=0}=(0,0).
\end{array}
\right.
\end{equation}

The well-posedness of the above systems are trivial. Note that $\rho_0 \geq c_0>0$, therefore $\theta^0_1,\theta^{u,0}_1$ satisfy the heat equations. whose well-posedness theory is classical. For $\theta^0_2,\theta^{u,0}_2$, they enjoy a linear parabolic system, whose well-posedness theory can be found in \cite{Xin2}, therefore here we do not discuss this topic for details.

In addition, the decay properties (or weighted estimates) of the correctors will be used in Section \ref{sec4} to establish the error bounds for the approximate solutions. The solvability of (\ref{2.3}) and (\ref{2.4}) with the decay estimates of the correctors will be discussed in Appendix \ref{ap}.

\section{Approximate solutions}\label{sec3}

To derive the error bounds for the approximate solutions of Navier-Stokes equations, we modify (\ref{2.1}) to ensure the boundary conditions exactly. This modification also can be found in \cite{Mazzucato,Temam,Temam1,XWang}.

Let $\psi(z)$ be a smooth function such that
\begin{equation}\label{3.0}
\left \{
\begin{array}{lll}
0 \leq \psi(z) \leq 1, \\
\psi(z) \equiv 1,\ z \in[0,\frac{1}{3}],\\
\psi(z) \equiv 0,\ z \in[\frac{1}{2},1].
\end{array}
\right.
\end{equation}
For $z\in[0,1]$, we have $\psi(z)\psi(1-z)\equiv 0$.

Note that there is no boundary layer for density, then we define
\begin{equation}\label{3.1}
\left \{
\begin{array}{lll}
\tilde{u}^{a}_1&:=u^{0}_1(z)+\psi(z)\theta^0_1\left(t;\frac{z}{\sqrt{\varepsilon}}\right)
+\psi(1-z)\theta^{u,0}_1\left(t;\frac{1-z}{\sqrt{\varepsilon}}\right),\\
\tilde{u}^{a}_2&:=u^{0}_2(t;x,z)+\psi(z)\theta^0_2\left(t;x,\frac{z}{\sqrt{\varepsilon}}\right)
+\psi(1-z)\theta^{u,0}_2\left(t;x,\frac{1-z}{\sqrt{\varepsilon}}\right).
\end{array}
\right.
\end{equation}
Recall that $\rho^\varepsilon \equiv \rho_0(z)$, therefore some simple calculations yield that the truncated approximations satisfy
\begin{equation}\label{3.3}
\begin{array}{lll}
\rho_0 \partial_t \tilde{u}^a_1-\varepsilon \partial_{z}^2\tilde{u}^a_1=\rho_0 f_1+A+B,
\end{array}
\end{equation}
\begin{equation}\label{3.4}
\begin{array}{lll}
\rho_0 \partial_t \tilde{u}^a_2-\varepsilon \Delta_{x,z}\tilde{u}^a_2+\rho_0\tilde{u}^a_1\partial_x \tilde{u}^a_2=\rho_0f_2+C+D+E,
\end{array}
\end{equation}
where the remainders are given by

\begin{equation}\label{3.8}
\begin{array}{lll}
A=&\sqrt{\varepsilon}\Big[\psi(z)Z\partial_z\rho_0(0)\partial_t \theta^0_1-\psi(1-z)Z^u\partial_z \rho_0(1)\partial_t \theta^{u,0}_1\\
&-2\psi'(z) \partial_Z\theta^0_1-2\psi'(1-z)\partial_{Z^u} \theta^{u,0}_1\Big],
\end{array}
\end{equation}
\begin{equation}\label{3.9}
\begin{array}{lll}
B&=\varepsilon[-\partial_z^2 u^0_1-\psi''(z)\theta^0_1-\psi''(1-z)\theta^{u,0}_1],
\end{array}
\end{equation}
\begin{equation}\label{3.10}
\begin{array}{lll}
C=&\psi(z)(\psi(z)-1)\rho_0\theta^0_1\partial_x\theta^0_2+\psi(1-z)(\psi(1-z)-1)\rho_0\theta^{u,0}_1\partial_x\theta^{u,0}_2,
\end{array}
\end{equation}
\begin{equation}\label{3.11}
\begin{array}{lll}
D=&\sqrt{\varepsilon}\Bigg[\psi(z)Z\bigg(\partial_z \rho_0(0)\partial_t \theta^0_2+\partial_z(\rho_0u^0_1)(t;0)\partial_x \theta^0_2\\
&+\partial_z(\rho_0\partial_x u^0_2)(t;x,0)\theta^0_1+\partial_z\rho_0(0)\theta^0_1\partial_x \theta^0_2\bigg)\\
&-\psi(1-z)Z^u\bigg(\partial_z \rho_0(1)\partial_t \theta^{u,0}_2+\partial_z(\rho_0 u^0_1)(t;1)\partial_x \theta^{u,0}_2\\
&+\partial_z(\rho_0\partial_x u^0_2)(t;x,1)\theta^{u,0}_1+\partial_z \rho_0(1)\theta^{u,0}_1\partial_x \theta^{u,0}_2\bigg)\\
&-2\psi'(z) \partial_Z \theta^0_2-2\psi'(1-z) \partial_{Z^u} \theta^{u,0}_2\Bigg],
\end{array}
\end{equation}
\begin{equation}\label{3.12}
\begin{array}{lll}
E=&\varepsilon\Big(-\psi(z)\partial_x^2 \theta^0_2-\psi(1-z) \partial_x^2 \theta^{u,0}_2\\
 &-\partial_x^2 u^0_2 -\partial_z^2 u^0_2-\psi''(z)\theta^0_2-\psi''(1-z)\theta^{u,0}_2\Big).
\end{array}
\end{equation}

The corresponding initial and boundary conditions are given as follows
\begin{equation}\label{3.13}
\left \{
\begin{array}{lll}
\tilde{u}^a|_{t=0}=u_0,\\
\tilde{u}^a|_{z=i}=\alpha^i, \ i=0,1.
\end{array}
\right.
\end{equation}

Every term of the approximate solution $\tilde{u}^a$ is determined by the Euler equations (\ref{2.21}) and the Prandtl-type effective equations (\ref{2.3})-(\ref{2.4}), therefore the approximate solution is well-defined.

\section{Proof of Theorem \ref{thm1}: Error estimates and convergence rates}\label{sec4}

In this section, we will prove our main result, i.e., the energy estimates of the error solutions between the approximation $\tilde{u}^a$ and the exact solution of the Navier-Stokes equations $u^\varepsilon$, which will give the convergence rates of the desired estimates.

We denote the error solutions by
\begin{equation}\label{4.1}
\left \{
\begin{array}{lll}
u^{err}_1(t;z)=u^\varepsilon_1(t;z)-\tilde{u}^a_1(t;z),\\
u^{err}_2(t;x,z)=u^\varepsilon_2(t;x,z)-\tilde{u}^a_2(t;x,z),
\end{array}
\right.
\end{equation}
then we obtain
\begin{equation}\label{4.2}
\left \{
\begin{array}{lll}
\rho_0 \partial_t u^{err}_1-\varepsilon \partial_{zz} u^{err}_1=-(A+B),\\
\rho_0 \partial_t u^{err}_2 -\varepsilon \Delta_{x,z} u^{err}_2+\rho_0(u^{err}_1\partial_x \tilde{u}^a_2+u^\varepsilon_1\partial_x u^{err}_2)=-(C+D+E),
\end{array}
\right.
\end{equation}
with the following boundary and inial conditions
\begin{equation}\label{4.3}
\left \{
\begin{array}{lll}
u^{err}|_{z=i}=0, \ i=0,1,\\
u^{err}|_{t=0}=0,
\end{array}
\right.
\end{equation}
where $A,B,C,D,E$ are defined in Section \ref{sec3}.

As we introduced in above sections, the existence and regularity of $(\rho^\varepsilon,u^\varepsilon)$, defined in (\ref{1.62})-(\ref{1.64}) and $(\tilde{\rho}^a,\tilde{u}^a)$ well-definded by (\ref{3.1}), ensure that of $u^{err}$. Therefore, it is enough to derive the uniform bound, independent of $\varepsilon$, to complete Theorem \ref{thm1}.

Before proving our main result, we introduce the anisotropic Sobolev inequality that will be used in the proof of Theorem \ref{thm1}. See \cite{Mazzucato} for instance.
\begin{lemma}\label{lemma4.1}{\bf (\cite{Mazzucato})}
There holds that
\begin{equation}\label{4.4}
\begin{aligned}
\Vert u \Vert_{L^\infty((0,T) \times \Omega)} &\leq C \big(\Vert u \Vert_{L^\infty(0,T;L^2(\Omega))}^{\frac{1}{2}}\Vert \partial_z u \Vert_{L^\infty(0,T;L^2(\Omega))}^{\frac{1}{2}}\\
&\quad +\Vert \partial_z u \Vert_{L^\infty(0,T;L^2(\Omega))}^{\frac{1}{2}}\Vert \partial_x u \Vert_{L^\infty(0,T;L^2(\Omega))}^{\frac{1}{2}} \\
&\quad + \Vert u \Vert_{L^\infty(0,T;L^2(\Omega))}^{\frac{1}{2}}\Vert \partial_x \partial_z u \Vert_{L^\infty(0,T;L^2(\Omega))}^{\frac{1}{2}}\big),
\end{aligned}
\end{equation}
for all $u \in H^1_0(\Omega)$. It is pointed out that the left-hand sides of the inequality could be infinite.
\end{lemma}

Now we are on the position to prove Theorem \ref{thm1}.
\begin{proof}[Proof of Theorem \ref{thm1}]
We will prove Theorem \ref{thm1} by the following steps.

\textbf{Step 1: Estimates for $u^{err}_1$.}

Multiplying (\ref{4.2})$_1$ by $u^{err}_1$, note that $\rho_0$ is time-independent, we apply integrating by parts over $\Omega$ to get that
\begin{equation}\label{4.5}
\begin{aligned}
\frac{1}{2}&\frac{\mathrm{d}}{\mathrm{d}t} \Vert \sqrt{\rho_0}u^{err}_1 \Vert_{L^2(0,1)}^2+\varepsilon \Vert \partial_z u^{err}_1 \Vert_{L^2(0,1)}^2\\
=&-\int_0^1 (A+B)u^{err}_1 \mathrm{d}z\\
\leq &\frac{2}{\sqrt{c_0}}\int_{1/3}^{2/3} \left|\sqrt{\varepsilon}(\partial_Z\theta^0_1+\partial_{Z^u}\theta^{u,0}_1)+\varepsilon (\theta^0_1+\theta^{u,0}_1)\right||\sqrt{\rho_0}u^{err}_1|\mathrm{d}z\\
&+\frac{\varepsilon}{\sqrt{c_0}} \Vert u^0_1\Vert_{H^2(0,1)}\Vert \sqrt{\rho_0}u^{err}_1\Vert_{L^2(0,1)}\\
&+\frac{C\varepsilon^{3/4}}{\sqrt{c_0}}\Vert \partial_z \rho_0\Vert_{L^\infty}\Vert \sqrt{\rho_0}u^{err}_1\Vert_{L^2(0,1)}\\
&\times\left(\Vert \langle Z\rangle \partial_t \theta^0_1\Vert_{L^2(0,\infty)}+\Vert \langle Z^u\rangle \partial_t \theta^{u,0}_1\Vert_{L^2(0,\infty)}\right)\\
\leq &C \Vert \sqrt{\rho_0}u^{err}_1\Vert_{L^2(0,1)}\bigg[\varepsilon^{7/4}\bigg(\Vert \langle Z\rangle^2\partial_Z\theta^0_1\Vert_{L^2(0,\infty)}\\
&+\Vert \langle Z^u\rangle^2\partial_{Z^u}\theta^{u,0}_1\Vert_{L^2(0,\infty)}\bigg)\\
&+\varepsilon\left(\Vert \theta^0_1\Vert_{L^2(0,\infty)}+\Vert \theta^{u,0}\Vert_{L^2(0,\infty)}+\Vert u^0_1\Vert_{H^2(0,1)}\right)\\
&+\varepsilon^{3/4}\left(\Vert \langle Z\rangle \partial_t \theta^0_1\Vert_{L^2(0,\infty)}+\Vert \langle Z^u\rangle \partial_t \theta^{u,0}_1\Vert_{L^2(0,\infty)}\right)\bigg],
\end{aligned}
\end{equation}
where some terms of the right-hand side of the last inequality can be estimated by the following way:
\begin{equation}\nonumber
\begin{aligned}
\int_{\frac{1}{3}}^{\frac{2}{3}}\left|\partial_Z\theta^0_1 u^{err}_1 \right|\mathrm{d}z\leq &C\varepsilon \Vert \sqrt{\rho_0}u^{err}_1\Vert_{L^2(0,1)}\left(\int_{\frac{1}{3\sqrt{\varepsilon}}}^{\frac{2}{3\sqrt{\varepsilon}}} \langle Z\rangle^4|\partial_Z\theta^0_1|^2\sqrt{\varepsilon}\mathrm{d}Z\right)^{1/2}\\
\leq &C \varepsilon^{\frac{5}{4}} \Vert \sqrt{\rho_0} u^{err}_1\Vert_{L^2(0,1)}\Vert \langle Z \rangle^2\partial_Z \theta^0_1\Vert_{L^2(0,\infty)},
\end{aligned}
\end{equation}
in which we have used the fact that $\rho_0\geq c_0>0$.

Applying the Young inequality and Gronwall inequality, we get that
\begin{equation}\label{4.6}
\begin{aligned}
\Vert\sqrt{\rho_0}u^{err}_1\Vert_{L^\infty(0,T;L^2(0,1))}+\sqrt{\varepsilon}\Vert \partial_z u^{err}_1\Vert_{L^2(0,T;L^2(0,1))} \leq  C\varepsilon^{\frac{3}{4}},
\end{aligned}
\end{equation}
which gives that
\begin{equation}\label{4.7}
\begin{aligned}
\Vert u^{err}_1\Vert_{L^\infty(0,T;L^2(0,1))} \leq  C\varepsilon^{\frac{3}{4}},
\end{aligned}
\end{equation}
here we have used the fact that $\rho_0\geq c_0>0$ again.

To obtain the estimate for $\partial_z u^{err}_1$, recall that $\rho_0\geq c_0>0$, then we rewrite (\ref{4.2})$_1$ as
\begin{equation}\label{4.8}
 \partial_t u^{err}_1-\frac{\varepsilon }{\rho_0}\partial_{zz} u^{err}_1=-\frac{(A+B)}{\rho_0}.\\
\end{equation}
Multiplying (\ref{4.8}) by $-\partial_z^2 u^{err}_1$, integrating by parts in $\Omega,$ one has
\begin{equation}\label{4.9}
\begin{aligned}
\frac{1}{2}&\frac{\mathrm{d}}{\mathrm{d}t} \Vert \partial_z u^{err}_1 \Vert_{L^2(0,1)}^2+\varepsilon \left\Vert \frac{\partial_z^2 u^{err}_1}{\sqrt{\rho_0}} \right\Vert_{L^2(0,1)}^2\\
=&\int_0^1 \frac{(A+B)\partial_z^2u^{err}_1}{\rho_0} \mathrm{d}z\\
\leq &C \varepsilon^{3/4}\left\Vert \frac{\partial_z^2u^{err}_1}{\sqrt{\rho_0}}\right\Vert_{L^2(0,1)}\\
&\times \bigg[\bigg(\Vert \langle Z\rangle^2\partial_Z\theta^0_1\Vert_{L^2(0,\infty)}+\Vert \langle Z^u\rangle^2\partial_{Z^u}\theta^{u,0}_1\Vert_{L^2(0,\infty)}\bigg)\\
&+\left(\Vert \theta^0_1\Vert_{L^2(0,\infty)}+\Vert \theta^{u,0}\Vert_{L^2(0,\infty)}+\Vert u^0_1\Vert_{H^2(0,1)}\right)\\
&+\left(\Vert \langle Z\rangle \partial_t \theta^0_1\Vert_{L^2(0,\infty)}+\Vert \langle Z^u\rangle \partial_t \theta^{u,0}_1\Vert_{L^2(0,\infty)}\right)\bigg],
\end{aligned}
\end{equation}
by the Young inequality and Gronwall inequality, we have
\begin{equation}\label{4.10}
\begin{aligned}
\Vert \partial_z u^{err}_1\Vert_{L^\infty(0,T;L^2(0,1))}+\sqrt{\varepsilon}\left\Vert \frac{\partial_z^2 u^{err}_1}{\sqrt{\rho_0}}\right\Vert_{L^2(0,T;L^2(0,1))} \leq C\varepsilon^{1/4}.
\end{aligned}
\end{equation}

Therefore, we deduce that
\begin{equation}\label{4.11}
\begin{aligned}
\Vert  u^{err}_1\Vert_{L^\infty(0,T;L^2(0,1))} &\leq C\varepsilon^{3/4},\\
\Vert  u^{err}_1\Vert_{L^\infty(0,T;H^1(0,1))} &\leq C\varepsilon^{1/4},\\
\Vert u^{err}_1 \Vert_{L^\infty ((0,R)\times (0,1))}&\leq \Vert  u^{err}_1\Vert_{L^\infty(0,T;L^2(0,1))}^{1/2}\Vert  u^{err}_1\Vert_{L^\infty(0,T;H^1(0,1))}^{1/2}\\
&\leq C \sqrt{\varepsilon}.
\end{aligned}
\end{equation}

\textbf{Step 2: Estimates for $u^{err}_2$.}

Multiplying (\ref{4.2})$_2$ by $u^{err}_2$, note that $\rho_0$ is time-independent, we apply integrating by parts over $\Omega$ to get that
\begin{equation}\label{4.12}
\begin{aligned}
\frac{1}{2}&\frac{\mathrm{d}}{\mathrm{d}t} \Vert \sqrt{\rho_0}u^{err}_2 \Vert_{L^2(\Omega)}^2+\varepsilon \Vert \nabla_{x,z} u^{err}_2 \Vert_{L^2(\Omega)}^2\\
=&-\int_\Omega \rho_0u^{err}_1\partial_x \tilde{u}^a_2u^{err}_2-\int_\Omega (C+D+E)u^{err}_2\\
=&:I_1+I_2+I_3+I_4.
\end{aligned}
\end{equation}
We estimate every $I_i(i=1,2,3,4)$ as follows.
\begin{equation}\label{4.13}
\begin{aligned}
I_1 \leq &\frac{C\varepsilon^{3/4}}{\sqrt{c_0}}\Vert \rho_0\Vert_{L^\infty}\Vert \sqrt{\rho_0}u^{err}_2 \Vert_{L^2(\Omega)}\\
&\cdot\left(\Vert \partial_x u^0_2\Vert_{L^\infty}+\Vert \partial_x \theta^0_2\Vert_{L^\infty(\Omega_\infty)}+\Vert \partial_x \theta^{u,0}_2\Vert_{L^\infty(\Omega_\infty)}\right)\\
\leq &C\varepsilon^{3/4}\Vert \sqrt{\rho_0}u^{err}_2 \Vert_{L^2(\Omega)},
\end{aligned}
\end{equation}
\begin{equation}\label{4.14}
\begin{aligned}
I_2 \leq &\frac{C\varepsilon}{\sqrt{c_0}}\Vert \rho_0\Vert_{L^\infty}\Vert \sqrt{\rho_0}u^{err}_2 \Vert_{L^2(\Omega)}\\
&\cdot\left(\Vert \theta^0_1\Vert_{L^\infty}\Vert\langle Z\rangle^2\partial_x \theta^0_2\Vert_{L^2}+\Vert \theta^{u,0}_1\Vert_{L^\infty}\Vert\langle Z^u\rangle^2\partial_x \theta^{u,0}_2\Vert_{L^2}\right)\\
\leq &C\varepsilon\Vert \sqrt{\rho_0}u^{err}_2 \Vert_{L^2(\Omega)},
\end{aligned}
\end{equation}
\begin{equation}\label{4.15}
\begin{aligned}
I_3 \leq&\frac{C\varepsilon^{3/4}}{\sqrt{c_0}}\Vert \sqrt{\rho_0}u^{err}_2\Vert_{L^2(\Omega)}\bigg[\Vert \partial_z \rho_0\Vert_{L^\infty}\Vert \langle Z\rangle\partial_t \theta^0_2\Vert_{L^2}\\
&+(\Vert\partial_z\rho_0\Vert_{L^\infty}\Vert u^0_1\Vert_{L^\infty}+\Vert \rho_0\Vert_{L^\infty}\Vert \partial_z u^0_1\Vert_{L^\infty})\Vert \langle Z\rangle \partial_x \theta^0_2\Vert_{L^2}\\
&+(\Vert \partial_z\rho_0\Vert_{L^\infty}\Vert \partial_x u^0_2\Vert_{L^\infty}+\Vert \rho_0\Vert_{L^\infty}\Vert \partial_{zx}u^0_2\Vert_{L^\infty})\Vert \langle Z\rangle\theta^0_1\Vert_{L^2}\\
&+\Vert \partial_z\rho_0\Vert_{L^\infty}\Vert \theta^0_1\Vert_{L^\infty}\Vert \langle Z\rangle \partial_x \theta^0_2\Vert_{L^2}+\Vert \partial_z \rho_0\Vert_{L^\infty}\Vert \langle Z^u\rangle\partial_t \theta^{u,0}_2\Vert_{L^2}\\
&+(\Vert\partial_z\rho_0\Vert_{L^\infty}\Vert u^0_1\Vert_{L^\infty}+\Vert \rho_0\Vert_{L^\infty}\Vert \partial_z u^0_1\Vert_{L^\infty})\Vert \langle Z^u\rangle \partial_x \theta^{u,0}_2\Vert_{L^2}\\
&+(\Vert \partial_z\rho_0\Vert_{L^\infty}\Vert \partial_x u^0_2\Vert_{L^\infty}+\Vert \rho_0\Vert_{L^\infty}\Vert \partial_{zx}u^0_2\Vert_{L^\infty})\Vert \langle Z^u\rangle\theta^{u,0}_1\Vert_{L^2}\\
&+\Vert \partial_z\rho_0\Vert_{L^\infty}\Vert \theta^{u,0}_1\Vert_{L^\infty}\Vert \langle Z^u\rangle \partial_x \theta^{u,0}_2\Vert_{L^2}+\Vert \partial_Z \theta^0_2\Vert_{L^2}+\Vert \partial_{Z^u} \theta^{u,0}_2\Vert_{L^2}\bigg],
\end{aligned}
\end{equation}
\begin{equation}\label{4.16}
\begin{aligned}
I_4 \leq&\frac{C\varepsilon}{\sqrt{c_0}}\Vert \sqrt{\rho_0}u^{err}_2\Vert_{L^2(\Omega)}\big(\Vert \partial_x^2\theta^0_2\Vert_{L^2}+\Vert \partial_x^2\theta^{u,0}_2\Vert_{L^2}+\Vert u^0_2\Vert_{H^2}\\
&+\Vert \theta^0_2\Vert_{L^2}+\Vert \theta^{u,0}_2\Vert_{L^2}\big).
\end{aligned}
\end{equation}
Putting the above estimates into (\ref{4.12}), using the Young inequality and Gronwall inequality, we have
\begin{equation}\label{4.17}
\begin{aligned}
\Vert \sqrt{\rho_0}u^{err}_2 \Vert_{L^\infty(0,T;L^2(\Omega))}+\sqrt{\varepsilon} \Vert \nabla_{x,z} u^{err}_2 \Vert_{L^2(0,T;L^2(\Omega))} \leq C\varepsilon^{3/4},
\end{aligned}
\end{equation}
which give from $\rho_0\geq c_0>0$ that
\begin{equation}\label{4.18}
\begin{aligned}
\Vert u^{err}_2 \Vert_{L^\infty(0,T;L^2(\Omega))} \leq C\varepsilon^{3/4}.
\end{aligned}
\end{equation}

Similarly, multiplying (\ref{4.2})$_2$ by $-\partial_{xx}u^{err}_2$, integrating by parts over $\Omega$, one has
\begin{equation}\label{4.19}
\begin{aligned}
\frac{1}{2}&\frac{\mathrm{d}}{\mathrm{d}t} \Vert \sqrt{\rho_0}\partial_x u^{err}_2 \Vert_{L^2(\Omega)}^2+\varepsilon \Vert \nabla_{x,z} \partial_x u^{err}_2 \Vert_{L^2(\Omega)}^2\\
=&\int_\Omega \rho_0u^{err}_1\partial_x^2 \tilde{u}^a_2\partial_x u^{err}_2-\int_\Omega \partial_x (C+D+E)\partial_xu^{err}_2\\
=&:I_5+I_6+I_7+I_8.
\end{aligned}
\end{equation}
We estimate every $I_i(i=5,6,7,8)$ as follows.
\begin{equation}\label{4.20}
\begin{aligned}
I_5 \leq &\frac{C\varepsilon^{3/4}}{\sqrt{c_0}}\Vert \rho_0\Vert_{L^\infty}\Vert \sqrt{\rho_0}\partial_x u^{err}_2 \Vert_{L^2(\Omega)}\\
&\cdot\left(\Vert \partial_x^2 u^0_2\Vert_{L^\infty}+\Vert \partial_x^2 \theta^0_2\Vert_{L^\infty(\Omega_\infty)}+\Vert \partial_x^2 \theta^{u,0}_2\Vert_{L^\infty(\Omega_\infty)}\right)\\
\leq &C\varepsilon^{3/4}\Vert \sqrt{\rho_0}\partial_x u^{err}_2 \Vert_{L^2(\Omega)},
\end{aligned}
\end{equation}
\begin{equation}\label{4.21}
\begin{aligned}
I_6 \leq &\frac{C\varepsilon}{\sqrt{c_0}}\Vert \rho_0\Vert_{L^\infty}\Vert \sqrt{\rho_0}\partial_x u^{err}_2 \Vert_{L^2(\Omega)}\\
&\cdot\left(\Vert \theta^0_1\Vert_{L^\infty}\Vert\langle Z\rangle^2\partial_x^2 \theta^0_2\Vert_{L^2}+\Vert \theta^{u,0}_1\Vert_{L^\infty}\Vert\langle Z^u\rangle^2\partial_x^2 \theta^{u,0}_2\Vert_{L^2}\right)\\
\leq &C\varepsilon\Vert \sqrt{\rho_0}\partial_x u^{err}_2 \Vert_{L^2(\Omega)},
\end{aligned}
\end{equation}
\begin{equation}\label{4.22}
\begin{aligned}
I_7 \leq&\frac{C\varepsilon^{3/4}}{\sqrt{c_0}}\Vert \sqrt{\rho_0}\partial_x u^{err}_2\Vert_{L^2(\Omega)}\bigg[\Vert \partial_z \rho_0\Vert_{L^\infty}\Vert \langle Z\rangle\partial_t \partial_x \theta^0_2\Vert_{L^2}\\
&+(\Vert\partial_z\rho_0\Vert_{L^\infty}\Vert u^0_1\Vert_{L^\infty}+\Vert \rho_0\Vert_{L^\infty}\Vert \partial_z u^0_1\Vert_{L^\infty})\Vert \langle Z\rangle \partial_x^2 \theta^0_2\Vert_{L^2}\\
&+(\Vert \partial_z\rho_0\Vert_{L^\infty}\Vert \partial_{xx} u^0_2\Vert_{L^\infty}+\Vert \rho_0\Vert_{L^\infty}\Vert \partial_{zxx}u^0_2\Vert_{L^\infty})\Vert \langle Z\rangle\theta^0_1\Vert_{L^2}\\
&+\Vert \partial_z\rho_0\Vert_{L^\infty}\Vert \theta^0_1\Vert_{L^\infty}\Vert \langle Z\rangle \partial_x^2 \theta^0_2\Vert_{L^2}+\Vert \partial_z \rho_0\Vert_{L^\infty}\Vert \langle Z^u\rangle\partial_t\partial_x \theta^{u,0}_2\Vert_{L^2}\\
&+(\Vert\partial_z\rho_0\Vert_{L^\infty}\Vert u^0_1\Vert_{L^\infty}+\Vert \rho_0\Vert_{L^\infty}\Vert \partial_z u^0_1\Vert_{L^\infty})\Vert \langle Z^u\rangle \partial_x^2 \theta^{u,0}_2\Vert_{L^2}\\
&+(\Vert \partial_z\rho_0\Vert_{L^\infty}\Vert \partial_x^2 u^0_2\Vert_{L^\infty}+\Vert \rho_0\Vert_{L^\infty}\Vert \partial_{zxx}u^0_2\Vert_{L^\infty})\Vert \langle Z^u\rangle\theta^{u,0}_1\Vert_{L^2}\\
&+\Vert \partial_z\rho_0\Vert_{L^\infty}\Vert \theta^{u,0}_1\Vert_{L^\infty}\Vert \langle Z^u\rangle \partial_x^2 \theta^{u,0}_2\Vert_{L^2}+\Vert \partial_{Z}\partial_x \theta^0_2\Vert_{L^2}+\Vert \partial_{Z^u} \partial_x \theta^{u,0}_2\Vert_{L^2}\bigg],
\end{aligned}
\end{equation}
\begin{equation}\label{4.22}
\begin{aligned}
I_8 \leq&\frac{C\varepsilon}{\sqrt{c_0}}\Vert \sqrt{\rho_0}\partial_xu^{err}_2\Vert_{L^2(\Omega)}\big(\Vert \partial_x^3\theta^0_2\Vert_{L^2}+\Vert \partial_x^3\theta^{u,0}_2\Vert_{L^2}+\Vert u^0_2\Vert_{H^3}\\
&+\Vert \partial_x\theta^0_2\Vert_{L^2}+\Vert \partial_x\theta^{u,0}_2\Vert_{L^2}\big).
\end{aligned}
\end{equation}
Putting the above estimates into (\ref{4.19}), using the Young inequality and Gronwall inequality, we have
\begin{equation}\label{4.23}
\begin{aligned}
\Vert \sqrt{\rho_0}\partial_xu^{err}_2 \Vert_{L^\infty(0,T;L^2(\Omega))}+\sqrt{\varepsilon} \Vert \nabla_{x,z} \partial_x u^{err}_2 \Vert_{L^2(0,T;L^2(\Omega))} \leq C\varepsilon^{3/4},
\end{aligned}
\end{equation}
which gives from $\rho_0\geq c_0>0$ that
\begin{equation}\label{4.24}
\begin{aligned}
\Vert \partial_x u^{err}_2 \Vert_{L^\infty(0,T;L^2(\Omega))} \leq C\varepsilon^{3/4}.
\end{aligned}
\end{equation}

To obtain the estimate for $\partial_z u^{err}_2$, we find that
\begin{equation}\label{4.25}
\begin{aligned}
 \partial_t u^{err}_2 -\frac{\varepsilon}{\rho_0} \Delta_{x,z} u^{err}_2+(u^{err}_1\partial_x \tilde{u}^a_2+u^\varepsilon_1\partial_x u^{err}_2)=-\frac{1}{\rho_0}(C+D+E).
\end{aligned}
\end{equation}
Multiplying (\ref{4.25}) by $-\partial_{zz}u^{err}_2$, integrating by parts over $\Omega$ to yield that
\begin{equation}\label{4.26}
\begin{aligned}
\frac{1}{2}\frac{\mathrm{d}}{\mathrm{d}t}&\Vert \partial_z u^{err}_2\Vert_{L^2}^2+\varepsilon\left\Vert \frac{1}{\sqrt{\rho_0}}\nabla_{x,z}\partial_z u^{err}_2\right\Vert_{L^2}^2\\
=&\varepsilon \int \frac{\partial_z \rho_0}{\rho_0^2}\partial_x u^{err}_2\partial_{xz}u^{err}_2\\
&+\int \frac{1}{\rho_0}(C+D+E) \partial_{zz}u^{err}_2+\int (u^{err}_1\partial_x \tilde{u}^a_2+u^\varepsilon_1\partial_x u^{err}_2)\partial_{zz}u^{err}_2\\
=&:I_9+I_{10}+I_{11}+I_{12}+I_{13}+I_{14}.
\end{aligned}
\end{equation}

Each term $I_i(i=9,10,11,12,13,14)$ can be estimates as follows.
\begin{equation}\label{4.26}
\begin{aligned}
I_9 \leq &\frac{C \varepsilon^{7/4}}{c_0^{3/2}}\Vert \partial_z \rho_0\Vert_{L^\infty}\left\Vert \frac{1}{\sqrt{\rho_0}}\partial_{xz}u^{err}_2\right\Vert_{L^2},
\end{aligned}
\end{equation}
\begin{equation}\label{4.27}
\begin{aligned}
I_{10} \leq &\frac{C\varepsilon}{\sqrt{c_0}}\Vert \rho_0\Vert_{L^\infty}\left\Vert \frac{1}{\sqrt{\rho_0}}\partial_{zz}u^{err}_2\right\Vert_{L^2}\\
&\cdot\left(\Vert \theta^0_1\Vert_{L^\infty}\Vert\langle Z\rangle^2\partial_x \theta^0_2\Vert_{L^2}+\Vert \theta^{u,0}_1\Vert_{L^\infty}\Vert\langle Z^u\rangle^2\partial_x \theta^{u,0}_2\Vert_{L^2}\right),
\end{aligned}
\end{equation}
\begin{equation}\label{4.28}
\begin{aligned}
I_{11} \leq&\frac{C\varepsilon^{3/4}}{\sqrt{c_0}}\left\Vert \frac{1}{\sqrt{\rho_0}}\partial_{zz}u^{err}_2\right\Vert_{L^2}\bigg[\Vert \partial_z \rho_0\Vert_{L^\infty}\Vert \langle Z\rangle\partial_t \theta^0_2\Vert_{L^2}\\
&+(\Vert\partial_z\rho_0\Vert_{L^\infty}\Vert u^0_1\Vert_{L^\infty}+\Vert \rho_0\Vert_{L^\infty}\Vert \partial_z u^0_1\Vert_{L^\infty})\Vert \langle Z\rangle \partial_x \theta^0_2\Vert_{L^2}\\
&+(\Vert \partial_z\rho_0\Vert_{L^\infty}\Vert \partial_x u^0_2\Vert_{L^\infty}+\Vert \rho_0\Vert_{L^\infty}\Vert \partial_{zx}u^0_2\Vert_{L^\infty})\Vert \langle Z\rangle\theta^0_1\Vert_{L^2}\\
&+\Vert \partial_z\rho_0\Vert_{L^\infty}\Vert \theta^0_1\Vert_{L^\infty}\Vert \langle Z\rangle \partial_x \theta^0_2\Vert_{L^2}+\Vert \partial_z \rho_0\Vert_{L^\infty}\Vert \langle Z^u\rangle\partial_t \theta^{u,0}_2\Vert_{L^2}\\
&+(\Vert\partial_z\rho_0\Vert_{L^\infty}\Vert u^0_1\Vert_{L^\infty}+\Vert \rho_0\Vert_{L^\infty}\Vert \partial_z u^0_1\Vert_{L^\infty})\Vert \langle Z^u\rangle \partial_x \theta^{u,0}_2\Vert_{L^2}\\
&+(\Vert \partial_z\rho_0\Vert_{L^\infty}\Vert \partial_x u^0_2\Vert_{L^\infty}+\Vert \rho_0\Vert_{L^\infty}\Vert \partial_{zx}u^0_2\Vert_{L^\infty})\Vert \langle Z^u\rangle\theta^{u,0}_1\Vert_{L^2}\\
&+\Vert \partial_z\rho_0\Vert_{L^\infty}\Vert \theta^{u,0}_1\Vert_{L^\infty}\Vert \langle Z^u\rangle \partial_x \theta^{u,0}_2\Vert_{L^2}+\Vert \partial_Z \theta^0_2\Vert_{L^2}+\Vert \partial_{Z^u} \theta^{u,0}_2\Vert_{L^2}\bigg],
\end{aligned}
\end{equation}
\begin{equation}\label{4.29}
\begin{aligned}
I_{12} \leq&\frac{C\varepsilon}{\sqrt{c_0}}\left\Vert \frac{1}{\sqrt{\rho_0}}\partial_{zz}u^{err}_2\right\Vert_{L^2}\big(\Vert \partial_x^2\theta^0_2\Vert_{L^2}+\Vert \partial_x^2\theta^{u,0}_2\Vert_{L^2}+\Vert u^0_2\Vert_{H^2}\\
&+\Vert \theta^0_2\Vert_{L^2}+\Vert \theta^{u,0}_2\Vert_{L^2}\big),
\end{aligned}
\end{equation}
\begin{equation}\label{4.30}
\begin{aligned}
I_{13} \leq& C\varepsilon^{3/4}\Vert \rho_0\Vert_{L^\infty}^{1/2}\left\Vert \frac{1}{\sqrt{\rho_0}}\partial_{zz}u^{err}_2\right\Vert_{L^2}\\
&\cdot\left(\Vert \partial_x u^0_2\Vert_{L^\infty}+\Vert \partial_x \theta^0_2\Vert_{L^\infty(\Omega_\infty)}+\Vert \partial_x \theta^{u,0}_2\Vert_{L^\infty(\Omega_\infty)}\right),
\end{aligned}
\end{equation}
\begin{equation}\label{4.31}
\begin{aligned}
I_{14} \leq& C\varepsilon^{3/4}\Vert \rho_0\Vert_{L^\infty}^{1/2}\Vert u^\varepsilon_1\Vert_{L^\infty}\left\Vert \frac{1}{\sqrt{\rho_0}}\partial_{zz}u^{err}_2\right\Vert_{L^2}.
\end{aligned}
\end{equation}
Putting the above estimates into (\ref{4.26}), applying the Young inequality and Gronwall inequality, we have
\begin{equation}\label{4.32}
\begin{aligned}
\Vert \partial_z u^{err}_2\Vert_{L^\infty(0,T;L^2(\Omega))}+\sqrt{\varepsilon}\left\Vert \frac{1}{\sqrt{\rho_0}}\nabla_{x,z}\partial_z u^{err}_2\right\Vert_{L^2(0,T;L^2(\Omega))}\leq C\varepsilon^{1/4}.
\end{aligned}
\end{equation}

In the above estimates, the bound with rate $\varepsilon^{1/4}$ can not be improved since we cannot perform any integration by parts for the higher derivatives with the direction with $z$, as $\partial_z u^{err}_2$ man not vanish on boundaries.

Similar arguments give that
\begin{equation}\label{4.33}
\begin{aligned}
\Vert \partial_{xx} u^{err}_2\Vert_{L^\infty(0,T;L^2(\Omega))}\leq C\varepsilon^{3/4},\\
\Vert \partial_{zx} u^{err}_2\Vert_{L^\infty(0,T;L^2(\Omega))}\leq C\varepsilon^{1/4}.
\end{aligned}
\end{equation}

Therefore, we deduce that
\begin{equation}\label{4.34}
\begin{aligned}
\Vert u^{err}_2\Vert_{L^\infty(0,T;H^1(\Omega))} \leq C\varepsilon^{\frac{1}{4}}.
\end{aligned}
\end{equation}

Finally, we use the Lemma \ref{lemma4.1} to get
\begin{equation}\label{4.52}
\begin{aligned}
\Vert u^{err}_2 \Vert_{L^\infty((0,T) \times \Omega)} &\leq C \bigg(\Vert u^{err}_2 \Vert_{L^\infty(0,T;L^2(\Omega))}^{\frac{1}{2}}\Vert \partial_z u^{err}_2 \Vert_{L^\infty(0,T;L^2(\Omega))}^{\frac{1}{2}}\\
&\quad +\Vert \partial_z u^{err}_2 \Vert_{L^\infty(0,T;L^2(\Omega))}^{\frac{1}{2}}\Vert \partial_x u^{err}_2 \Vert_{L^\infty(0,T;L^2(\Omega))}^{\frac{1}{2}} \\
&\quad + \Vert u^{err}_2 \Vert_{L^\infty(0,T;L^2(\Omega))}^{\frac{1}{2}}\Vert \partial_x \partial_z u^{err}_2 \Vert_{L^\infty(0,T;L^2(\Omega))}^{\frac{1}{2}}\bigg)\\
&\leq C\sqrt{\varepsilon}.
\end{aligned}
\end{equation}

Combine the above steps, we have
\begin{equation}\label{4.56}
\begin{aligned}
\Vert u^{err} \Vert_{L^\infty(0,T;L^2(\Omega))} &\leq C\varepsilon^\frac{3}{4},\\
\Vert u^{err} \Vert_{L^\infty(0,T;H^1(\Omega))} &\leq C\varepsilon^\frac{1}{4},\\
\Vert u^{err} \Vert_{L^\infty((0,T) \times \Omega)}& \leq C\sqrt{\varepsilon},
\end{aligned}
\end{equation}
which completes the proof of Theorem \ref{thm1}.
\end{proof}

\appendix
\section{Well-posedness and weighted estimates of the equations for the correctors}\label{ap}

In this Appendix, we discuss the solvability of the Prandtl-type effective equations for the correctors appeared in Section \ref{sec2}. Thanks to the similarity of the lower and upper correctors, we only need to consider the correctors near the boundary $z=0$, i.e., we only focus on $(\theta^0_1,\theta^0_2)$ here. Note that in the equations for the correctors, we work in $(x,Z)\in \Omega_\infty=[0,L]\times [0,\infty)$.

We start with  $\theta^0_1$, which satisfies the following initial boundary value problem for a heat equation
\begin{equation}\label{a.75}
\left \{
\begin{array}{lll}
\rho_0(0)\partial_t \theta^0_1 -\partial_{ZZ} \theta^0_1=0,\\
\theta^0_1|_{Z=0}=\beta^0_1(t)-u^0_1(t;0),\\
\theta^0_1|_{Z=\infty}=0,\\
\theta^0_1|_{t=0}=0,
\end{array}
\right.
\end{equation}
where $\rho_0(z)$ is independent of $t$.

Denote
$$\langle Z \rangle:=\sqrt{1+Z^2}.$$
Recall that
$$\rho_0 \geq c_0>0,$$
therefore the equation for $\theta^0_1$ is non degenerate.

Based on the above facts, we easily derive the following conclusions.

\begin{lemma}\label{lemma6.1}
Suppose $\beta^0_1 \in L^\infty(0,T), u^0_1 \in  L^\infty(0,T;H^m(0,1)),$ and, $\rho_0 \in H^m(0,1),$ where $m>5$. Then for any $l \in \mathbb{N}$, there holds that
\begin{equation}\label{a.76}
\Vert \langle Z \rangle^l \theta^0_1 \Vert_{L^\infty((0,T)\times \Omega_\infty)} \leq C,
\end{equation}
\begin{equation}\label{a.761}
\Vert \langle Z \rangle^l \theta^0_1 \Vert_{L^\infty(0,T;L^2( \Omega_\infty))}+\Vert \langle Z \rangle^l \partial_Z \theta^0_1 \Vert_{L^2(0,T;L^2( \Omega_\infty))}\leq C,
\end{equation}
\begin{equation}\label{a.763}
\Vert \langle Z \rangle^l \partial_Z \theta^0_1 \Vert_{L^\infty(0,T;L^2( \Omega_\infty))}+\Vert \langle Z \rangle^l \partial_t \theta^0_1 \Vert_{L^2(0,T;L^2( \Omega_\infty))}\leq C,
\end{equation}
where the constant $C>0$ depends on
$$\Vert \beta^0_1\Vert_{H^m(0,T)}, \Vert u^0_1 \Vert_{ L^\infty (0,T;H^m(0,1))}, \Vert \rho_0 \Vert_{H^m(0,1)},c_0,l, T,$$
but it is independent of $\varepsilon$.
\end{lemma}
\begin{proof}
Note that $\rho^0(t;x,0) +\sigma^0 \geq c_0 >0$, then the proof of (\ref{a.76}) is straightforward from the maximum principle of the parabolic equations, which is similar to the proof of the Lemma A.1 of \cite{Mazzucato} and we omit the details here.

Set
$$\overline{\theta^0_1}:=\theta^0_1-e^{-Z}(\beta^0_1(t)-u^0_1(t;0)),$$
then we have
\begin{equation}\label{a.751}
\left \{
\begin{array}{lll}
\rho_0(0)\partial_t \overline{\theta^0_1}-\partial_{ZZ} \overline{\theta^0_1}=
e^{-Z}[\partial_t \beta^0_1(t)-(\beta^0_1(t)-u^0_1(t;0))],\\
\overline{\theta^0_1}|_{Z=0,\infty}=0,\\
\overline{\theta^0_1}|_{t=0}=0.
\end{array}
\right.
\end{equation}

Multiplying (\ref{a.751})$_1$ by $\langle Z \rangle^{2l}\overline{\theta^0_1}$, integrating by parts over $\Omega_\infty$ and adding the results to give that
\begin{equation}\label{a.77}
\begin{aligned}
\frac{1}{2} \frac{\mathrm{d}}{\mathrm{d}t} &\int_{\Omega_\infty} \rho_0(0) \langle Z \rangle^{2l} |\overline{\theta^0_1}|^2+\int_{\Omega_\infty} \langle Z \rangle^{2l} |\partial_Z \overline{\theta^0_1}|^2\\
&=\int_{\Omega_\infty} e^{-Z}[\partial_t \beta^0_1(t)-(\beta^0_1(t)-u^0_1(t;0))]\langle Z \rangle^{2l}\overline{\theta^0_1}\\
& \leq C \Vert \sqrt{\rho_0(0)}  \langle Z \rangle^{l} \overline{\theta^0_1}\Vert_{L^2},
\end{aligned}
\end{equation}
where we have used the Sobolev embedding and the fact that $\langle Z \rangle^l e^{-Z}$ is uniformly bounded in $Z$ for any $l$.

Using Young's inequality and Gronwall's inequality, we have
\begin{equation}\label{a.78}
\begin{aligned}
\Vert \langle Z \rangle^l \overline{\theta^0_1} \Vert_{L^\infty(0,T;L^2( \Omega_\infty))}+\Vert \langle Z \rangle^l \partial_Z \overline{\theta^0_1} \Vert_{L^2(0,T;L^2( \Omega_\infty))}\leq C,
\end{aligned}
\end{equation}
in which the fact that $\rho_0 \geq c_0$ has been used. Therefore (\ref{a.761}) follows.

Multiplying (\ref{a.751})$_2$ by $\langle Z \rangle^{2l}\partial_{t}\overline{\theta^0_1}$, integrating by parts on $\Omega_\infty$ to yield that
\begin{equation}\label{a.802}
\begin{aligned}
 \int_{\Omega_\infty} \rho_0(0) &\langle Z \rangle^{2l} |\partial_t \overline{\theta^0_1}|^2+\frac{1}{2} \frac{\mathrm{d}}{\mathrm{d}t} \Vert \langle  Z \rangle^l \partial_Z \overline{\theta^0_1} \Vert_{L^2(\Omega_\infty)}^2\\
 &=\int_{\Omega_\infty} e^{-Z}[\partial_t \beta^0_1(t)-(\beta^0_1(t)-u^0_1(0))]\langle Z \rangle^{2l}\partial_t\overline{\theta^0_1}\\
 &\quad -\int_{\Omega_\infty}\partial_t \overline{\theta^0_1} \partial_Z \overline{\theta^0_1}  \partial_Z(\langle Z \rangle^{2l}) \\
 &\leq C(\Vert \beta^0_1\Vert_{H^m(0,T)}, \Vert u^0_1 \Vert_{H^m(0,1)})\Vert \langle  Z \rangle^l \partial_t \overline{\theta^0_1} \Vert_{L^2(\Omega_\infty)}\\
 &\quad +C\Vert \sqrt{\rho_0(0)} \langle Z\rangle^l \partial_t \theta^0_1 \Vert_{L^2}\Vert  \langle Z \rangle^{l}\partial_Z\overline{\theta^0_1}\Vert_{L^2}.
\end{aligned}
\end{equation}
Using (\ref{a.78}), the Young's inequality and the Gronwall's inequality, one has
\begin{equation}\label{a.803}
\Vert \langle Z \rangle^l \partial_Z \overline{\theta^0_1} \Vert_{L^\infty(0,T;L^2( \Omega_\infty))}+\Vert \langle Z \rangle^l \partial_t \overline{\theta^0_1} \Vert_{L^2(0,T;L^2( \Omega_\infty))}\leq C,
\end{equation}
hence (\ref{a.763}) holds.

This completes the proof.
\end{proof}

Now we consider the same problem for $\theta^0_2$. Consider
\begin{equation}\label{a.81}
\left \{
\begin{array}{lll}
\rho_0(0)\partial_t \theta^0_2 -\partial_{ZZ} \theta^0_2+(u^0_1(t;0)+\theta^0_1)\rho_0(0)\partial_x \theta^0_2 \\
\quad \quad \quad \quad \quad +\rho_0(0)\partial_x u^0_2(t;x,0)\theta^0_1=0,\\
\theta^0_2|_{Z=0}=\beta^0_2(t;x)-u^0_2(t;x,0),\\
\theta^0_2|_{Z=\infty}=0,\\
\theta^0_2|_{t=0}=0.
\end{array}
\right.
\end{equation}

We recall here that $\rho_0 (0) \geq c_0$ for $c_0>0$. The well-posedness of (\ref{a.81}) can be obtained by the similar techniques in Section 4 of \cite{Xin1}. One can also refer to \cite{Mazzucato,Han} for instance.

Our attentions are paid to the weighted estimates of $\theta^0_2$, which are used in Section \ref{sec4}.
\begin{lemma}\label{lemma6.2}
Suppose $\beta^0_2 \in H^2(0,T;H^m(\partial{\Omega})), u^0_2 \in L^\infty(0,T; H^m(\Omega))$, $\rho_0 \in H^m(0,1)$, where $m>5$. Then for any $l \in \mathbb{N}$, there holds that
\begin{equation}\label{a.82}
\Vert \langle Z \rangle^l \partial_x^i \theta^0_2 \Vert_{L^\infty((0,T)\times \Omega_\infty)} \leq C,\ i=0,1,2,
\end{equation}
\begin{equation}\label{a.83}
\Vert \langle Z \rangle^l \partial_x^i \theta^0_2 \Vert_{L^\infty(0,T;L^2( \Omega_\infty))}+\Vert \langle Z \rangle^l \partial_x^i\partial_Z \theta^0_2 \Vert_{L^2(0,T;L^2( \Omega_\infty))}\leq C,\ i=0,1,2,
\end{equation}
\begin{equation}\label{a.84}
\Vert \langle Z \rangle^l \partial_x^i \partial_Z \theta^0_2 \Vert_{L^\infty(0,T;L^2( \Omega_\infty))}+\Vert \langle Z \rangle^l \partial_x^i \partial_t \theta^0_2 \Vert_{L^2(0,T;L^2( \Omega_\infty))}\leq C,\ i=0,1,2,
\end{equation}
where the constant $C>0$ depends on
$$\Vert( \beta^0_1,\beta^0_2) \Vert_{H^2(0,T;H^m(\partial{\Omega}))}, \Vert u^0 \Vert_{L^\infty(0,T; H^m(\Omega))}, \Vert \rho_0 \Vert_{H^m(\Omega)},c_0,l, T,$$
but it is independent of $\varepsilon$.
\end{lemma}

\begin{proof}
The estimates of this lemma can be obtained by the standard energy methods, which is similar to that of Lemma \ref{lemma6.1} and the Appendix A of \cite{Mazzucato}. We can modify the techneque of \cite{Mazzucato} to complete the proof. The readers can see \cite{Mazzucato} for more details and we omit it here.
\end{proof}

Applying the similar arguments on $(\theta^{u,0}_1,\theta^{u,0}_2),$ we have the following results.
\begin{lemma}
Under the assumptions of Lemma \ref{lemma6.1} and Lemma \ref{lemma6.2}, we have
\begin{equation}\label{a.85}
\Vert \langle Z^u \rangle^l \theta^{u,0}_1 \Vert_{L^\infty((0,T)\times \Omega_\infty)} \leq C,
\end{equation}
\begin{equation}\label{a.851}
\begin{aligned}
\Vert \langle Z^u \rangle^l \theta^{u,0}_1 \Vert_{L^\infty(0,T;L^2( \Omega_\infty))}
+\Vert \langle Z^u \rangle^l \partial_Z \theta^{u,0}_1 \Vert_{L^2(0,T;L^2( \Omega_\infty))}\leq C,
\end{aligned}
\end{equation}
\begin{equation}\label{a.853}
\Vert \langle Z^u \rangle^l \partial_Z \theta^{u,0}_1 \Vert_{L^\infty(0,T;L^2( \Omega_\infty))}+\Vert \langle Z^u \rangle^l \partial_t \theta^{u,0}_1 \Vert_{L^2(0,T;L^2( \Omega_\infty))}\leq C,
\end{equation}
\begin{equation}\label{a.854}
\Vert \langle Z^u \rangle^l \partial_x^i \theta^{u,0}_2 \Vert_{L^\infty((0,T)\times \Omega_\infty)} \leq C,\ i=0,1,2,
\end{equation}
\begin{equation}\label{a.855}
\Vert \langle Z^u \rangle^l \partial_x^i \theta^{u,0}_2 \Vert_{L^\infty(0,T;L^2( \Omega_\infty))}+\Vert \langle Z^u \rangle^l \partial_x^i\partial_Z \theta^{u,0}_2 \Vert_{L^2(0,T;L^2( \Omega_\infty))}\leq C,\ i=0,1,2,
\end{equation}
\begin{equation}\label{a.856}
\Vert \langle Z^u \rangle^l  \partial_x^i \partial_Z \theta^{u,0}_2 \Vert_{L^\infty(0,T;L^2( \Omega_\infty))}+\Vert \langle Z^u \rangle^l \partial_x^i \partial_t \theta^{u,0}_2 \Vert_{L^2(0,T;L^2( \Omega_\infty))}\leq C, \ i=0,1,2,
\end{equation}
where the constant $C>0$ depends on
$$\Vert \beta^0_1\Vert_{H^m(0,T)}, \Vert u^0_1\Vert_{L^\infty(0,T;H^m(0,1))}, \Vert \beta^0_2 \Vert_{H^2(0,T;H^m(\partial{\Omega}))}, \Vert u^0_2 \Vert_{L^\infty(0,T; H^m(\Omega))},c_0,$$
$$\Vert \rho_0 \Vert_{ H^m(0,1)},l, T,$$
but it is independent of $\varepsilon$.
\end{lemma}

\smallskip
{\bf Acknowledgment.}

Ding's research is supported by the National Natural Science Foundation of China (No.11371152, No.11571117, No.11871005 and No.11771155) and Guangdong Provincial Natural Science Foundation (No.2017A030313003). Lin's research is supported by the Innovation Project of Graduate School of South China Normal University (No.2018LKXM009).
Niu's research is supported by the National Natural Science Foundation of China (No.11471220 and No.11871046) while she was visiting at the Institute of Mathematical Sciences of the Chinese University of Hong Kong.

\medskip
\newpage

\end{document}